\documentclass[a4,12pt]{article}%

\usepackage{graphicx}
\usepackage{graphics}

\usepackage{amsmath}
\usepackage{amsfonts}
\usepackage{amssymb}
\usepackage{enumerate}
\usepackage{xcolor}

\newtheorem{theorem}{Theorem}[section]

\newtheorem{definition}[theorem]{Definition}
\newtheorem{example}[theorem]{Example}

\newtheorem{lemma}[theorem]{Lemma}

\newtheorem{proposition}[theorem]{Proposition}

\newtheorem{remark}[theorem]{Remark}
\newtheorem{remarks}[theorem]{Remarks}

\newenvironment{proof}[1][Proof]{\textbf{#1.} }{\ \rule{0.5em}{0.5em}}

\newcommand{\E}{{\rm \bf E}}
\newcommand{\F}{{\rm \bf F}}

\newcommand{\lcp}{{\rm LCP}}

\newcommand{\prob}{{\rm \bf P}}
\newcommand{\rmd}{{\rm d}}

\newcommand{\supp}{{\rm supp}}

\newcommand{\ep}{\varepsilon}

\newcommand{\R}{\mathbb{R}}
\newcommand{\N}{\mathbb{N}}

\newcommand{\A}{\mathbb{A}}
\newcommand{\dN}{\N}

\newcommand{\dR}{\R}

\newcounter{figurecounter}
\newcounter{figure:table}
\newcounter{figure:table1}
\newcounter{figure:game}
\newcounter{figure:arrows}
\newcounter{figure:loops}
\newcounter{figure:ftv}
\newcounter{figure:ftv:2}

\begin{document}

\title{Absorption Paths and Equilibria\\
     in Quitting Games%
\thanks{
G. Ashkenazi-Golan, I. Krasikov, and E. Solan acknowledge the support of the Israel Science Foundation, grants \#217/17 and \#722/18,
and NSFC-ISF Grant \#2510/17.
G. Ashkenazi-Golan, C. Rainer, and E. Solan acknowledge the support of the COST action 16228, the European Network of Game Theory.}}

\author{Galit Ashkenazi-Golan%
\thanks{The School of Mathematical Sciences, Tel Aviv
University, Tel Aviv 6997800, Israel. e-mail: galit.ashkenazi@gmail.com.},
Ilia Krasikov%
\thanks{National Research University Higher School of Economics, 20 Myasnitskaya Ulitsa, Moscow 101000, Russia. e-mail: krasikovis.main@gmail.com.},
Catherine Rainer,%
\thanks{Univ Brest, UMR CNRS 6205, 6, avenue Victor-le-Gorgeu, B.P. 809, 29285 Brest cedex, France.
e-mail: Catherine.Rainer@univ-brest.fr.}
and Eilon Solan%
\thanks{The School of Mathematical Sciences, Tel Aviv
University, Tel Aviv 6997800, Israel. e-mail: eilons@post.tau.ac.il.}}

\maketitle

\begin{abstract}
We study quitting games
and define the concept of \emph{absorption paths},
which is an alternative definition to strategy profiles that accomodates both discrete time aspects and continuous time aspects,
and is parameterized by the total probability of absorption in past play rather than by time.
We then define the concept of \emph{sequentially 0-perfect} absorption paths,
which are shown to be limits of $\ep$-equilibrium strategy profiles as $\ep$ goes to 0.
We finally identify a class of quitting games that possess sequentially 0-perfect absorption paths.
\end{abstract}



\section{Introduction}

Stochastic games were introduced by Shapley (1953)
as a dynamic model, where the players' behavior affects the evolution of the state variable.
Whether every multiplayer stochastic game admits an $\ep$-equilibrium is one of the most difficult open
problems in game theory to date.
Mertens and Neyman (1981) proved that the value exists in two-player zero-sum games,
Vieille (2000a, 2000b) proved that an $\ep$-equilibrium exists in two-player nonzero-sum games,
Solan (1999) extended this result to three-player absorbing games,
and Flesch, Schoenmakers, and Vrieze (2008, 2009) proved the existence of an $\ep$-equilibrium when each player
controls one component of the state variable.

Solan and Vieille (2001) introduced a new class of stochastic games, called \emph{quitting games},
where each player has two actions, continue and quit,
the game terminates once at least one player chooses quit,
and the terminal payoff depends on the set of players who choose to quit at the termination stage.
Solan and Vieille (2001) proved that if the payoff function satisfies a certain condition,
then an $\ep$-equilibrium exists.
Simon (2007, 2012) and Solan and Solan (2020) extended this result to other families of payoff functions.
Though the class of quitting games is simple --
if the game has not terminated by a given stage, then necessarily all players continued so far
-- the analysis of these games is intricate,
the mathematical tools used to study them are diverse,
and include dynamical systems, topological tools, and linear complementarity problems,
and the equilibria these games possess may be complex
(see, Flesch, Thuijsman, and Vrieze (1997), Solan (2003), and Solan and Vieille (2002)).

The main difficulty in studying $\ep$-equilibria in stochastic games
is that the undiscounted payoff is not continuous over the space of strategies,
hence one cannot apply a fixed point theorem to prove the existence of an $\ep$-equilibrium.
In this paper we provide a new representation for strategy profiles in quitting games, termed \emph{absorption paths}.
This representation allows for both discrete-time aspects and continuous-time aspects in the players' behavior.
Moreover, the undiscounted payoff is continuous over the space of absorption paths.
In fact, the space of absorption paths is a compactification of the space of absorbing strategy profiles.

We define the concept of \emph{sequentially 0-perfect} absorption paths,
which are the analog of equilibria in standard strategy profiles.
We then show that limits of $\ep$-equilibria in standard strategy profiles
are sequentially 0-perfect absorption paths,
and that every sequentially 0-perfect absorption path induces an $\ep$-equilibrium in standard strategy profiles,
for every $\ep > 0$.
Finally, using Viability Theory we identify one class of quitting game where sequentially 0-perfect absorption paths exist.

The paper is organized as follows.
The model of quitting games is presented in Section~\ref{section:model},
and the equilibrium concept that we study is presented in Section~\ref{perfectness}.
Absorption paths are presented in Section~\ref{section:ap},
and their application to prove existence of $\ep$-equilibrium
in a certain class of quitting games
is described in Section~\ref{section:continuous}.
Concluding remarks appear in Section~\ref{section:discussion}.

\section{The Model}
\label{section:model}

\begin{definition}
A \emph{quitting game} is a
pair $\Gamma = (I,r)$,
where $I$ is a finite set of \emph{players} and
$r : \prod_{i \in I}\{C^i,Q^i\} \to \dR^I$ is a \emph{payoff function}.
\end{definition}

Player~$i$'s action set is $A^i := \{C^i,Q^i\}$.
These actions
are interpreted as continue and quit, respectively.
Set $A:=\prod_{i \in I} A^i$.
The game is played as follows.
At every stage $n \in \dN$ each player $i \in I$ chooses an action $a^i_n \in A^i$.
If all players continue, the play continues to the next stage;
if at least one player quits, the play terminates, and the terminal payoff is $r(a_n)$,
where $a_n = (a^i_n)_{i \in I}$.
If no player ever quits, the payoff is $r(\vec C)$, where $\vec C := (C^i)_{i \in I}$.

A \emph{mixed strategy profile} is a vector $\xi = (\xi^i)_{i \in I} \in [0,1]^I$,
with the interpretation that $\xi^i$ is the probability with which player~$i$ quits.
The probability of absorption under the mixed action profile $\xi$ is
$p(\xi) := 1-\prod_{i \in I}(1-\xi^i)$.
Extend the absorbing payoff to mixed action profiles that are absorbing with positive probability:
for every $\xi \in [0,1]^I$ such that $\xi \neq \vec 0$,
$r(\xi) := \frac{\sum_{a \in A^*} \xi(a) r(a)}{p(\xi)}$,
where $\xi(a) := \left(\prod_{\{i \colon a^i = Q^i\}} \xi^i\right)\cdot\left(\prod_{\{i \colon a^i = C^i\}} (1-\xi^i)\right)$, for every $a \in A$.

A (behavior) \emph{strategy} of player~$i$ is a function $x^i = (x_n^i)_{n \in \dN} : \dN \to [0,1]$, with the interpretation
that $x_n^i$ is the probability that player~$i$ quits at stage $n$ if the game
did not terminate
before that stage.
A \emph{strategy profile} is a vector $x = (x^i)_{i \in I}$ of strategies, one for each player.

We denote by $A^* := A \setminus \{\vec C\}$ the set of all action profiles in which at least one player quits,
by $A_1^* := \{ (Q^i,C^{-i}), i \in N\}$ the set of all action profiles in which exactly one player quits,
where $C^{-i} := (C^j)_{j \neq i}$,
and by $A_{\geq 2}^* := A^* \setminus A_1$ the set of all action profiles in which at least two players quit.

Given a sequence $(a_n)_{n=1}^N$, which may be finite or infinite,
set $\theta := \min\{n \leq N \colon a_n \in A^*\}$, where the minimum over an empty set is $\infty$.
When finite, $\theta$ is the first stage in which at least one of the players quit.
In this case
let $I_* := \{ i \in I \colon {a_{\theta}^i} = Q^i\}$ be the set of players who quit at the terminal stage.

For every strategy profile $x$,
the probability distribution of the random variable $(\theta,a_\theta)$ is denoted $\prob_x$.
Denote by $\E_x$ the corresponding expectation operator.
A strategy profile $x$ is \emph{absorbing} if $\prob_x(\theta < \infty) = 1$.

The \emph{payoff} under strategy profile $x$ is
\[ \gamma(x) := \E_x\left[\mathbf{1}_{\{\theta < \infty\}} r(a_\theta) + \mathbf{1}_{\{\theta = \infty\}} r(\vec C)\right]. \]
Let $\ep \geq 0$.
A strategy profile $x^*$ is an \emph{$\ep$-equilibrium} if $\gamma^i(x^*) \geq \gamma^i(x^i,x^{*,-i}) - \ep$
for every player $i \in I$ and every strategy $x^i$ of player~$i$.

It is easy to check that every two-player quitting game admits an $\ep$-equilibrium,
for every $\ep > 0$.
Solan (1999) extended this result to three-player quitting games,
see also Flesch, Thuijsman, and Vrieze (1997).
Whether every quitting game admits
an $\ep$-equilibrium for every $\ep > 0$
is an open problem.

\section{Sequential $\ep$-Perfectness}
\label{perfectness}

\subsection{$\ep$-Perfectness in Strategic-Form Games}
\label{section:perfect}

Let $G = (I, (A^i)_{i \in I}, r)$ be a strategic-form game
with set of players $I$, set of actions $A^i$ for each player $i \in I$,
and payoff function $r : A \to \dR^I$,
where $A = \prod_{i \in I} A^i$.

In an $\ep$-equilibrium, no player can profit more than $\ep$ by deviating.
This does not rule out the possibility that a player plays with small probability an action that
generates
 her a low payoff.
This deficiency is taken care of by the following concept,
which requires that a player does not play with positive probability actions that generates her a low payoff.
\begin{definition}
Let $G = (I, (A^i)_{i \in I}, r)$ be a strategic-form game, let
$i \in I$, and let $\xi \in \prod_{i \in I} \Delta(A^i)$ be a mixed action profile.
Player~$i$ is \emph{$\ep$-perfect} at $\xi$ in $G$ if the following conditions hold for every action $a^i \in A^i$:
\begin{eqnarray}
\label{equ:91}
&&r^i(a^i,\xi^{-i}) \leq r^i(\xi) + \ep,\\
\label{equ:92}
&& \xi^i(a^i) > 0 \ \ \ \Longrightarrow \ \ \ r^i(a^i,\xi^{-i}) \geq r^i(\xi) - \ep.
\end{eqnarray}
\end{definition}

Eq.~\eqref{equ:91} means that player~$i$ cannot gain more than $\ep$ by unilaterally altering her action;
Eq.~\eqref{equ:92} demands that player~$i$
cannot lose more than $\ep$ no matter which one of the actions
to which she assigns positive probability is played.

Standard continuity arguments yield that if
player~$i$ is $\ep_{k}$-perfect at a mixed action profile $\xi_{k}$ in the game $G_{k} = (I,r_k)$,
if $(\xi_{k})_{k \in \dN}$ converges to a limit $\xi$,
if $(\ep_{k})_{k \in \dN}$ converges to 0,
and if $r$ is a payoff function that satisfies $r^i = \lim_{k\to \infty} r^i_k$,
then player~$i$ is 0-perfect at $\xi$ in $G = (I,r)$.

\subsection{Sequentially $\ep$-Perfect Players in Quitting Games}

In this section we extend the concept of $\ep$-perfect players to quitting games.
Consider a quitting game $\Gamma = (I,r)$.
For every vector $y \in \dR^I$ let $G_\Gamma(y)$ be the one shot game with set of players $I$,
set of actions $A^i = \{Q^i,C^i\}$ for each player $i \in I$, and payoff function $r_\Gamma$ defined by
\[ r_\Gamma(y;a) := \left\{
\begin{array}{lll}
r(a) & \ \ \ \ \ \ & a \neq \vec C,\\
y & & a = \vec C.
\end{array}
\right. \]

The game $G_\Gamma(y)$ represents one stage of the game $\Gamma$,
when the continuation payoff is $y$.
A strategy profile in $G_\Gamma(y)$ is a vector $\xi \in [0,1]^I$,
with the interpretation that $\xi^i$ is the probability that player~$i$ chooses the action $Q^i$, for each $i \in I$.

We now define the concept of sequential $\ep$-perfectness in quitting games.
For every $n \in \dN$ denote by $\gamma_n(x)$ the expected payoff under $x$,
conditional that the game
did not
terminate in the first $n-1$ stages.%
\footnote{Note that since a strategy $x^i$ is a function from $\dN$ to $[0,1]$,
the conditional probability distribution $\prob_x(\cdot \mid \theta > n)$ is well defined even when $\prob_x(\theta \leq n) = 1$. }
    \[ \gamma_n(x) := \E_x[\mathbf{1}_{\{\theta < \infty\}} r(a_\theta) + \mathbf{1}_{\{\theta = \infty\}} r(\vec C) \mid \theta \geq n]. \]

\begin{definition}
Let $\Gamma$ be a quitting game
and let $i \in I$ be a player.
Player~$i$ is \emph{sequentially $\ep$-perfect} at the strategy profile $x$ in $\Gamma$ if for every $n \in \dN$,
player~$i$ is $\ep$-perfect at the mixed action profile $x_n$ in the strategic-form game $G_\Gamma(\gamma_{n+1}(x))$.
\end{definition}

\begin{remark}
In the strategic-form game $G_{\Gamma}(\gamma_{n+1}(x))$,
when the other players play $x^{-i}_n$,
the payoff of player~$i$ when she plays $x^i_n$ (resp.~$Q^i$, $C^i$) is
$\gamma_n^i(x)$ (resp.~$r^i(Q^i,x^{-i})$, $(1-p(C^i,x^{-i}_n)) \gamma_{n+1}(x) + p(C^i,x^{-i}_n) r^i(C^i,x^{-i}_n)$).
Therefore if player $i$ is $\ep$-perfect at $x_n$ in $G_{\Gamma}(\gamma_{n+1}(x))$,
then in particular
$r^i(Q^i,x_n^{-i})\leq \gamma_n^i(x)+\ep$,
and, if $x_n^i(Q^i)>0$ then
$r^i(Q^i,x_n^{-i})\geq \gamma_n^i(x)-\ep$.
\end{remark}

The following two results relate $\ep$-equilibria to
sequential $\ep$-perfectness in quitting games.

\begin{theorem}[Simon, 2007, Theorem 3 + Solan and Vieille, 2001, Proposition 2.13]
\label{theorem:simon}
Assume that the quitting game $\Gamma$ admits an $\ep$-equilibrium,
for every $\ep > 0$.
Then at least one of the following statements hold.
\begin{enumerate}
\item[(S.1)]   For every $\ep > 0$ sufficiently small the game admits a stationary $\ep$-equilibrium.
\item[(S.2)]   For every $\ep > 0$ sufficiently small the game admits an $\ep$-equilibrium $x$ that has the following structure:
there is a player $i \in I$ who quits with probability 1 at the first stage;
from the second stage and on, all players punish player~$i$ at her min-max level.\footnotemark
\item[(S.3)]   For every $\ep > 0$ sufficiently small there is an absorbing strategy profile $x$
such that all players $i \in I$ are sequentially $\ep$-perfect at $x$.
\end{enumerate}
\end{theorem}

\footnotetext{The min-max level of player~$i$ is $v^i := \inf_{x^{-i}}\sup_{x^i} \gamma^i(x^i,x^{-i})$.}

\begin{theorem}[Solan and Vieille, 2001, Propositions 2.4 and 2.13]
\label{theorem:sv}
Let $\ep > 0$ be sufficiently small.
Every absorbing strategy profile $x$ at which all players are sequentially $\ep$-perfect is an $\ep^{1/6}$-equilibrium.
\end{theorem}

\section{An Alternative Representation of Strategy Profiles}
\label{section:ap}

A strategy profile $x = (x_n)_{n \in \dN}$ is parameterized by
time: $x_n^i$ is the probability that player~$i$ quits at stage
$n$ if the game did not terminate before that stage.
As is well known, the space of
strategies is compact in the product topology.
There are two issues with this topology:
\begin{itemize}
\item
The payoff is
not continuous in this topology.
Indeed, if for every $k \in \dN$, $x^k$ is the
stationary strategy profile in which in every stage each player
quits with probability $\tfrac{1}{k}$, then the sequence
$(x^k)_{k \in \dN}$ converges to the strategy profile $x$
that always continues. While under the strategy profile $x^k$
absorption occurs with probability 1 and $\lim_{k \to \infty}
\gamma(x^k) = \tfrac{1}{|I|} \sum_{i \in I} r(Q^i,C^{-i})$,
under the strategy profile $x$ the game is never absorbed and
$\gamma(x) = r(\vec C)$.
\item
It may not be possible to generate the limit behavior of a sequence of strategy profiles by a strategy profile.
For example,
when $(x^k)$ are the strategy profiles that are defined in the first bullet,
we have
$\lim_{k \to \infty} \prob_{x^k}[a_\theta = (Q^i,C^{-i}) \mid \theta = n] = \frac{1}{|I|}$ for every $n \in \dN$,
yet there is no strategy profile $x$ that satisfies
$\prob_{x}[a_\theta = (Q^i,C^{-i}) \mid \theta = n] = \frac{1}{|I|}$ for every $n \in \dN$.
Indeed, under such a strategy profile $x = (x^i)_{i \in I}$,
for every $n \in \dN$ we have $x^i_n > 0$
for each $i \in I$,
and then $\sum_{i \in I} \prob_x[a_\theta = (Q^i,C^{-i}) \mid \theta = n] < 1$.
\end{itemize}

In this section we will provide an alternative representation of strategy profiles,
that takes care of both of these issues
by allowing
both discrete-time behavior and continuous-time behavior.
The representation will be based on a change of parametrization:
instead of parameterizing the
strategy profile according to time, we will parameterize it
according to the probability of termination. Thus, the parameter
$t$ will run from 0 to 1, and for every action profile $a \in A^*$
and every $t \in [0,1]$ we will indicate the probability by which
the game is absorbed by the action profile $a$ up to that moment
in which the total probability of absorption is $t$.

\subsection{Absorption Paths}

Let $\F$ be the set of c\`adl\`ag paths  $\pi=(\pi_t(a),a\in A^*)_{t\in[0,1]}$ with values in $[0,1]^{A^*}$,
such that, for all $a\in A^*$, $t\mapsto \pi_t(a)$ is nondecreasing. We endow $\F$  with the weak topology:
a sequence $(\pi^k)\subset\F$ converges to $\pi$ if
$\int_{[0,1]}f(t)d\pi^k_t(a)\to\int_{[0,1]}f(t)d\pi_t(a)$,
for every continuous map $f:[0,1]\to\R$ and every $a\in A^*$.
In such a case we write $\pi^k\Rightarrow\pi$.
Recall that $\pi^k\Rightarrow\pi$ if and only if $\pi^k_t\to\pi_t$ for every $t\in [0,1]$ where $\pi$ is continuous,
and that the set $\F$ is sequentially compact.

For each $\pi\in\F$, set
$\pi_{0-}(a):=0$ for every $a \in A^*$,
$\widehat\pi_t := \sum_{a\in A^*}\pi_t(a)$, and
$\Delta\pi_t := \pi_t-\pi_{t-}$ for every $t\in [0,1]$.
Set $T(\pi):=\{ t\in[0,1], \widehat\pi_t=t\}$, and  denote by $S(\pi)$ the set of jumps: $S(\pi)=\{ t\in[0,1], \Delta\pi_t\neq 0\}$.

Finally we introduce the right-hand side derivative of $t\mapsto\pi_t$ : for every $t\in[0,1)$
set $\dot{\pi}_t := \liminf_{s\searrow t}\frac{\pi_s-\pi_t}{s-t}$.
By Lebesgue's Theorem for the differentiability of monotone functions,
since $t \mapsto \pi_t(a)$ is nondecreasing for every $a \in A^*$,
the liminf is in fact a limit almost everywhere in $[0,1 )$.

\begin{definition}
    \label{defabs}
The set $\A$ of {\bf absorption paths} is the set of all paths $\pi\in\F$ such that the following hold.
\begin{enumerate}
    \item[(A.1)] For every $t\in[0,1]$, we have $\widehat\pi_t\geq t$.
    \item[(A.2)] On each connected component $(t_1,t_2)$ of $[0,1]\setminus (S(\pi)\cup T(\pi))$, $\widehat\pi$ is constant and equal to $t_2$.
    \item[(A.3)] For every $t\in S(\pi)$,
     there exists $\xi_t = (\xi^i_t)_{i\in N}\in[0,1]^I$ such that
\begin{equation}
\label{deltaxi}     \frac{\Delta\pi_t(a)}{1-t}=\xi_t(a), \ \ \ \forall a\in A^*.
\end{equation}
    \item[(A.4)]  For every $t\in T(\pi) \setminus\{ 1\}$ we have $\supp(\dot{\pi}_t)\subseteq A^*_1$.
    \end{enumerate}

\end{definition}

\begin{remarks}
    \label{const}
    Let $\pi\in\A$ be an absorption path.
    \begin{enumerate}
    \item
    For every $t \in S(\pi) \cup T(\pi)$,
    the quantity $\pi_t(a)$ should be thought of as the unconditional probability that the play is absorbed by
    the action profile $a$, until the moment in which the total probability of absorption is $t$.

    \item
    Elements
    $t \in S(\pi)$ correspond to play in discrete time, and for such $t$,
    $\xi_t$ is the mixed action profile the players play at $t$,
and $1-t$ is the total probability of absorption up to $t$.
    This explains (A.3).
    \item
    Elements
 $t \in T(\pi) \setminus \{1\}$ correspond to play in continuous time. This explains (A.4).
    \item
If $(t,t')$ is a connected component of
$[0,1] \setminus (S(\pi) \cup T(\pi))$,
then $t \in S(\pi)$ and $t' = t + (1-t)p(\xi_t)$.
This interval corresponds to the increase in probability due to play in discrete time.

    \item Since, for all $a\in A^*$,  $s\mapsto\pi_s(a)$ is nondecreasing,
    $\pi$ is continuous at $t$ if and only if $\widehat\pi$ is continuous at $t$,
    for every $t\in [0,1]$.
It follows from (A.2) that on each connected component of $[0,1]\setminus(S(\pi)\cup T(\pi))$ the process $\pi$ is constant.

    \item  Let $t\in S(\pi)$.
            Since $\pi$ is c\`adl\`ag and nondecreasing, we get from (A.1) that
            $\widehat \pi_t>t$,
            and from (A.2) that $\widehat\pi_{s}=\widehat\pi_{t}$ for every $s\in[t,\widehat\pi_t)$.
            In particular $\widehat\pi_{\widehat\pi_ t-}=\widehat\pi_t$.
\item For every $t\in[0,1]$, both $\widehat\pi_{t-}$ and $\widehat\pi_t$ belong to the set $T(\pi)\cup S(\pi)$.

\item
From (A.2) and Remark~\ref{const}(7),
we deduce that $[0,1)$ is partitioned to a countable number of intervals $U=[t_1,t_2)$,
with, either $U\subset T(\pi)$, or $t_1\in S(\pi)$ and $t_2=\widehat\pi_{t_1}$.
    On each of these intervals, $\pi$ is continuous, with $\widehat\pi_t=t$ if $U\subset T(\pi)$, and $\widehat\pi_t=t_2$ otherwise.

        \item The function $\pi$ is continuous at $t=1$: indeed, since, for all $t\in[0,1]$, $t\leq\widehat \pi_t\leq 1$,
        we have $\widehat\pi_1=\lim_{t\nearrow 1}\widehat\pi_t=1$.
\item
For every $a \in A^*_{\geq 2}$, the function $t \mapsto \pi_t(a)$ is piecewise constant.
\item
The reader may wonder why we defined $\dot\pi$ with liminf and not with limsup.
It turns out that to ensure that the set of absorption paths is sequentially compact
(see Proposition \ref{AP:compact} below),
we need to define $\dot\pi$ with liminf.
    \end{enumerate}

    \end{remarks}

\begin{example}
\label{example:1}
Figure~\arabic{figurecounter} displays an absorption path $\pi$ for the case $|I|=2$.
The interpretation of this absorption path is the following:
First Players~1 and~2 simultaneously quit with positive probability,
Player~1 with probability $\tfrac{1}{3}$ and Player~2 with probability $\tfrac{1}{4}$;
then Player~1 quits alone with probability $\tfrac{1}{2}$;
and then Players~1 and~2 quit together in continuous time, Player~1 with rate $\tfrac{1}{2}$
and Player~2 with rate $\tfrac{1}{4}$. We have $S(\pi) = \{0,\frac{1}{2}\}$ and $T(\pi) = [\frac{3}{4},1]$.

\begin{figure}
\center{
    \includegraphics[scale=0.25]{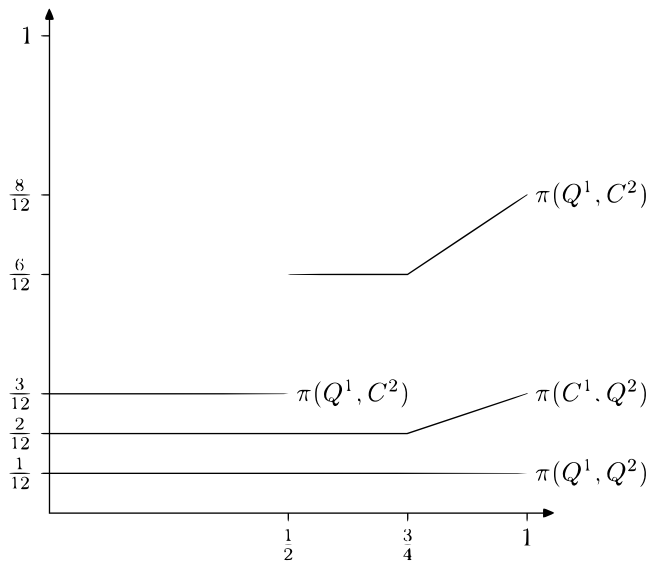}
   \caption{The absorption path in Example~\ref{example:1}.}}
\end{figure}

\addtocounter{figurecounter}{1}

\end{example}

\begin{remark}
\label{remark:strategy}
Every absorbing strategy profile
$x = (x_n)_{n \in \dN}$
naturally defines an absorption path $\pi^x$
that contains only discrete-time aspects.
Indeed, for every $n \in \dN$ denote $t_n := \prob_x(\theta < n)$,
and define
\[ \pi^x_t(a) := \prob_x(\theta \leq n, a_\theta = a), \ \ \ \forall a \in A^*, n \in \dN, t \in [t_n,t_{n+1}). \]
The reader can verify that
$S(\pi^x) = \{t_1,t_2,\ldots\}$, $T(\pi^x) = \{1\}$, and $\xi_{t_n} = x_n$ for every $n \in \dN$.
\end{remark}

\begin{remark}
The function $x \mapsto\pi^x$ that is defined in
Remark~\ref{remark:strategy} is not one-to-one. Indeed, fix an
absorbing strategy profile $x$ and let $x'$ be the strategy
profile in which all players continue in the first stage, and from
the second stage on they follow $x$ :
\[ x'^i_n =
\left\{
\begin{array}{lll}
C^i, & \ \ \ \ \ & \hbox{if } n = 1,\\
x^i_{n-1}, & & \hbox{if } n > 1.
\end{array}
\right.
\]
Then $\pi^{x'} = \pi^x$. In fact, given an
absorbing strategy profile $x$, the addition or elimination of
stages in which all players continue is the only way to create an
absorbing strategy profile $x'$ such that $\pi^x = \pi^{x'}$.
\end{remark}

The following result states that the
set of all $\pi^x$, where $x$ ranges over all absorbing strategy profile,
is dense in the set of absorption paths.
Thus, the set of absorption paths is a compactification of the set of absorbing strategy profiles.

\begin{proposition}
\label{prop:convergence}
For every absorption path $\pi$ there is a sequence of absorbing strategy
profiles $(x^k)_{k\in\dN}$ such that $\pi^{x^k}\Rightarrow\pi$.
\end{proposition}

To prove Proposition~\ref{prop:convergence} we need the following technical lemma.
\begin{lemma}
\label{lemma:technical}
Let $\ep > 0$
be sufficiently small,
and let $y \in \Delta(A)$ be a distribution that satisfies
$p(y)
:= 1 - y(\vec C)
\leq \ep$ and $y(a) \leq \ep y(Q^i,C^{-i})$
for each $i \in I$ and every $a \in A^*_{\geq 2}$ such that $a^i = Q^i$.
Let $\xi \in [0,1]^I$ be the unique mixed action profile that satisfies
$p(\xi) = p(y)$ and
\begin{equation}
\label{equ:t1}
\frac{\xi^i}{\xi^j} = \frac{y(Q^i,C^{-i})}{y(Q^j,C^{-j})}, \ \ \ \forall i,j \in I,
\end{equation}
where $\frac00=1$.
Then

\begin{equation}
\label{equ:technical}
|\xi(a) - y(a)| \leq 2^{|I|}
\cdot(|I|+1)\cdot
\ep p(y), \ \ \ \forall a \in A^*.
\end{equation}
\end{lemma}

\begin{proof}
For every $i \in I$ we have
$y(Q^i,C^{-i}) \leq p(y) \leq \ep$,
and similarly $\xi^i \leq \ep$.
This implies that $y(a), \xi(a) \in [0,\ep p(y)]$ for every $a \in A^*_{\geq 2}$,
hence Eq.~\eqref{equ:technical} holds for $a \in A^*_{\geq 2}$.
It follows that
\[ \left| \sum_{a \in A^*_1} \xi(a) - \sum_{a \in A^*_1} y(a) \right| \leq 2^{|I|}\ep p(y) \]
and
$0 \leq \xi^i - \xi(Q^i,C^{-i}) \leq 2^{|I|-1}\ep p(y)$ for every $i \in I$.
Hence
\[ \left| \sum_{i \in I} \xi^i - \sum_{a \in A^*_1} y(a) \right| \leq 2^{|I|}\cdot(|I|+1)\cdot\ep p(y). \]
Eq.~\eqref{equ:t1} implies now that Eq.~\eqref{equ:technical} holds for $a \in A^*_1$, provided $\ep$ is sufficiently small.
\end{proof}

\medskip

Note that $\xi$ in Lemma~\ref{lemma:technical} is uniquely defined, because
$\xi^{i} = z \cdot \frac{y(Q^i,C^{-i})}
{\sum_{j\in I} y(Q^j,C^{-j})}$,
where $z$ is determined so that $p(\xi) = p(y)$.

\medskip

\begin{proof}[Proof of Proposition~\ref{prop:convergence}]
The idea of the proof is to discretize $[0,1]$,
that is, for every $k\in\N$,
we define a countable set $S^k=(s^k_n)_{n\in\N}\subset[0,1]$ and a strategy profile $x^k$ in such a way that $x^k_n$ approximates
the behavior under $\pi$ between the $n$'th and $(n+1)$'st point of $S^k$.
The set $S^k$ contains the points $t$ in $S(\pi)$ where the conditional probability of quitting is larger than $\frac 1k$,
and covers $[0,1]$
minus the corresponding intervals $[t,\widehat\pi_t)$ with well chosen points $s^k_n$ such that $s^k_{n+1}\leq \frac{1}{k}(1-s^k_n)$,
i.e., the conditional probability of absorption
in $[s^k_n,s^k_{n+1})$
is less than $\frac 1k$.

We turn to the formal construction.
Fix an absorption path $\pi \in \A$ and $k \in \dN$.
Let
\[ S_0^k := \{ t \in S(\pi) \colon \widehat\pi_t - t \geq \tfrac{1-t}{k}\} = \{ t \in S(\pi) \colon p(\xi_t) \geq \tfrac{1}{k}\}. \]
Define the set $S^k =(s^k_n)_{n\in\dN}\subset[0,1]$ as follows:
\begin{itemize}
    \item $s^k_1:=0$.
    \item For $n\in\N$,
    define
inductively
$s^k_{n+1}:=\sup\left(\left( (S(\pi)\cup T(\pi))\cap[0,s^k_n+ \frac{1-s^k_n}k]\right)\cup\{ \widehat\pi^k_{s^k_n}\}\right)$.
In words,
if $s^k_n \in S^k_0$ then $s^k_{n+1} =
\widehat\pi_{s^k_n}$,
and if
$s^k_n \not\in S^k_0$,
then $s^k_{n+1}$ is the maximal point in $S(\pi)\cup T(\pi)$ smaller than $s^k_n + \frac{1-s^k_n}k$.
\end{itemize}

Define a strategy profile $x^k$ as follows:
\begin{enumerate}
\item[(D.1)]
If $s_n^k \in S_0^k$, set
$x^k_n := \xi_{s_n^k}$.
\item[(D.2)]
If $s_n^k \not\in S_0^k$, let
$x^{k}_n = (x^{k,i}_n)_{i \in I}$ be the unique solution of the following system of equations:
\begin{eqnarray}
\label{equ:88}
p(x^k_n) &=& 1 - \prod_{i \in I} (1-x^{k,i}_n) = \frac{s^k_{n+1} - s^k_n}{1-s^k_n},\\
\label{equ:89}
\frac{x^{k,i}_n}{x^{k,j}_n} &=& \frac{\pi_{s^k_{n+1}-}(Q^i,C^{-i}) -\pi_{s^k_{n}-}(Q^i,C^{-i})}{\pi_{s^k_{n+1}-}(Q^j,C^{-j}) -\pi_{s^k_{n}-}(Q^j,C^{-j})},
\hbox{ where } \frac{0}{0} = 1.
\end{eqnarray}
\end{enumerate}

Recall that as mentioned after Lemma~\ref{lemma:technical},
a unique solution to Eqs.~\eqref{equ:88}--\eqref{equ:89} exists.

The convergence $\pi^{x^k} \Rightarrow \pi$ will follow as soon as we show that
\begin{equation}
\label{equ:801}
\|
\pi^{x^k}_{s^k_{n}-} - \pi_{s^k_{n}-}\|_\infty \leq s^k_{n} \cdot 2^{|I|}
\cdot (|I|+1)
/k, \ \ \ \forall k \in \dN, \forall n \in \dN.
\end{equation}
Eq.~\eqref{equ:801} is trivially satisfied for $k=1$.
We shall suppose that the relation is true for some $n\in\N$ and prove that it still holds for $n+1$.
(D.1) and Eq.~\eqref{equ:88} ensure that
$\widehat\pi^{x^k}_{s^k_{n+1}-} - \widehat\pi^{x^k}_{s^k_{n}-}=
\widehat\pi_{s^k_{n+1}-} - \widehat\pi_{s^k_{n}-}$ :
for every $n\in\dN$,
the probability of absorption
at stage $n$
under the probability $\pi^{x^k}$,
is the same as under the original absorption path $\pi$
in $[s^k_n,s^k_{n+1})$.
This implies that $\widehat \pi_{s^k_n-}^{x^k} = \widehat \pi_{s^k_n-}$ for every $n \in \dN$.

If $s^k_n \in S^k_0$, then (D.1) implies that $s^k_{n+1} = \widehat\pi_{s^n_k-}$ and
$\pi^{x^k}_{s^k_{n}-}(a) - \pi^{x^k}_{s^k_{n}-}(a)=\pi_{s^k_{n+1}-}(a) - \pi_{s^k_{n+1}-}(a) $ for every $a \in A^*$,
and therefore Eq.~\eqref{equ:801} holds for every $n+1$.

Suppose now that $s^k_n \not\in S^k_0$.
Set
$y(a) := \frac{\pi_{s^k_{n+1}-}(a) - \pi_{s^k_{n}-}(a)}{1-s^k_n}$
for every $a \in A^*$
(and $y(\vec C) := 1-\sum_{a \in A^*}y(a)$).
Then $p(y) = \frac{s^k_{n+1}-s^k_n}{1-s^k_n}$.
By Lemma~\ref{lemma:technical},
$|x^k_n(a) - y(a)| < 2^{|I|}\cdot
(|I|+1) \cdot
p(y)/k$ for every $a \in A^*$.
Since $p(y) = \frac{s^k_{n+1}-s^k_n}{1-s^k_n}$ and
\[ \pi^{x^k}_{s^k_{n+1}-}(a) = \pi^{x^k}_{s^k_{n}-}(a) + (1-s^k_n)x^k_n(a), \ \ \
\pi_{s^k_{n+1}-}(a) = \pi_{s^k_{n}-}(a) + (1-s^k_n)y(a), \]
it follows that
\begin{eqnarray*}
| \pi^{x^k}_{s^k_{n+1}-}(a) - \pi_{s^k_{n+1}-}(a)|
&\leq& s^k_n \cdot 2^{|I|}
\cdot (|I|+1) \cdot
/k + (1-s^k_n)\frac{s^k_{n+1}-s^k_n}{1-s^k_n}2^{|I|}
\cdot (|I|+1) \cdot
/k\\
&=& s^k_{n+1}\cdot2^{|I|}
\cdot (|I|+1) \cdot
/k,
\end{eqnarray*}
as desired.
\end{proof}

\begin{remark}
\label{remark:11}
The behavior ``Player~1 quits with probability 1, and all other players continue throughout the game'' may be
translated in many ways to absorption paths. Here are some examples:
\begin{itemize}
\item   Player~1 quits with probability 1 in the first stage of
the game. In this case, we have $T(\pi) = \{1\}$ and $S(\pi)=\{ 0\}$ (Figure~\arabic{figurecounter}(a)).
\item   Player~1
quits with probability $\tfrac{1}{2}$ in each stage. In this case,
we have $T(\pi) =\{1\}$ and $S=
\{0,\tfrac{1}{2},\tfrac{3}{4},\tfrac{7}{8},\cdots\}$ (Figure~\arabic{figurecounter}(b)).
 \item
Player~1 ``quits continuously''. Here  $S(\pi)=\emptyset$, $T(\pi) = [0,1]$, and
$\pi_t(Q^1,C^{-1}) = t$, for every $t \in [0,1]$ (Figure~\arabic{figurecounter}(c)).
\item
And we may have combinations of the above (Figure~\arabic{figurecounter}(d)).
\end{itemize}

\begin{figure}
\center{
\includegraphics[scale=0.2]{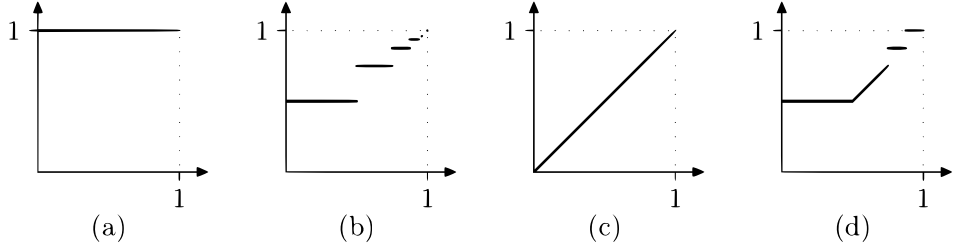}
   \caption{Four possibilities for the function $\pi_t(Q^1,C^{-1})$ in Remark~\ref{remark:11}.}}
\end{figure}

\addtocounter{figurecounter}{1}

\end{remark}

\begin{proposition}
\label{AP:compact}
The set of absorption paths $\A$ is sequentially compact:
for every sequence $(\pi^k)\in\A$ of absorption paths, there exists $\pi\in\A$ and a subsequence, still denote by $(\pi^k)$,
which converges weakly to $\pi$.
Moreover, this subsequence  can be chosen in such a way that
for every $t\in S(\pi)$, there are two sequences $(t_k)\subset[0,1]$ and $(\xi^k)\subset[0,1]^I$ with  $t_k\to t$ and $(\xi^k)\to\xi_{t}$ as $k\to\infty$,
and such that, for every $k\in\N$, $t_k\in S(\pi^k)$, and Eq.~\eqref{deltaxi} holds for $\pi^k$ and $\xi^k$ at $t_k$.
\end{proposition}

\begin{proof}
Let $(\pi^k)$ be a sequence of absorption paths.
Since  $\F$ is sequentially compact,
there exists a subsequence, still denote by $(\pi^k)$,
and $\pi\in\F$, such that $\pi^k\Rightarrow\pi$. We have to show that $\pi\in\A$.

Since $\pi_t^k\to\pi_t$ for a.e. $t\in[0,1]$,
it follows that $\widehat \pi_t^k\to\widehat \pi_t$ for a.e. $t\in[0,1]$,
and therefore (A.1) passes to the limit: $\widehat\pi_t\geq t$ for all $t\in[0,1]$.

To show that (A.2) holds for $\pi$,
let $U$ be a connected component of $[0,1]\setminus(T(\pi)\cup S(\pi))$.
Fix  $t\in U$. Since $\pi$ is continuous at $t$, we have
$\pi_t = \lim_{k \to \infty} \pi^k_t$.
Since $\widehat\pi_t>t$,
for every $\ep\in(0,\widehat\pi_t-t)$, there exists $k_0 \in \dN$ such that for every $k\geq k_0$
we have $\widehat\pi^k_t>\widehat\pi_t-\ep> t$.
Since $\pi^k$ belong to $\A$, it is constant on $[t,\widehat\pi_t-\ep)$. It follows that $\pi$ is also constant on $[t,\widehat\pi_t-\ep)$.
Since this is true for every $\ep>0$ sufficiently small, $\pi$ is constant on $[t,\widehat\pi_t)$,
and is equal to $\pi_t$.

We turn to prove that (A.3) holds for $\pi$.
Fix $t\in S(\pi)$.
There exists a subsequence of $(\pi^k)$, still denoted $(\pi^k)$,
and a sequence $(s_k)\subset [0,1]$ such that $s_k\to t$ and $\pi^k_{s_k}\to \pi_t$.
For each $k$, set $t_k:=\min\{ s\leq s_k,\pi^k_s=\pi^k_{s_k}\}$,
where the infimum is attained because of the right continuity of $\pi^k$.
Since $t \in S(\pi)$ we have
$\widehat\pi_t>t$,
hence $\widehat\pi^k_{t_k}>t_k$ for every $k$ sufficiently large.
By the definition of $t_k$ and (A.2), it follows that $t_k\in S(\pi^k)$.

We argue that $t_k \to t$.
Let $\widetilde t$ be an accumulation point of $(t_k)$. Since
$t_k \leq s_k\to t$,
we have $\widetilde t\leq t$.
If $\widetilde t<t$,
consider $s\in[\widetilde t,t)$ such that $\pi^k_{s}\to\pi_{s}$.
Then, for every $\ep>0$ and every $k$ large enough, we have
\[  \widehat \pi_t - \ep \leq \widehat\pi^k_{s_k} = \widehat\pi^k_{t_k}
\leq \widehat\pi_{s}+\ep\leq \widehat\pi_{t-}+\ep,\]
which is impossible for
$\ep < (\widehat \pi_t - \widehat \pi_{t-})/2$.

Since $t_k \to t$,
every accumulation point of $(\pi^k_{t_k-})$ belongs to the set $\{\pi_{t-},\pi_t\}$, and,
since $\lim_{k \to \infty}\widehat\pi^k_{t_k-}=\lim_{k \to \infty}s_k=t<\widehat\pi_t$,
it follows that $\lim_{k \to \infty}\pi^k_{t_k-}=\pi_{t-}$,
which implies that $\lim_{k \to \infty}\Delta\pi^k_{t_k}=\Delta\pi_t$.

For each $k \in \dN$, since
$t_k \in S(\pi^k)$,
there exists $\xi^k\in[0,1]^{I}$ such that
\begin{equation}
\label{equ:83}
\Delta\pi^k_{t_k}(a)=(1-t_k)\left(\prod_{\{i \colon a^i=Q^i\}}\xi^{k,i}\right)
\left(\prod_{\{i\colon a^i=C^i\}}(1-\xi^{k,i})\right), \ \ \  a\in A^*.
\end{equation}
        We can find a subsequence of $(t_k)$  and $\xi\in[0,1]^I$, such that $\xi^{k,i}\to\xi^i$ for all $i \in I$.
Taking the limit as $k\to\infty$ in Eq.~\eqref{equ:83} we get
\[ \Delta\pi_t(a)=(1-t)\left(\prod_{\{i \colon a^i=Q^i\}}\xi^i\right)
\left(\prod_{\{i\colon a^i=C^i\}}(1-\xi^{i})\right), \ \ \  a\in A^*.\]
This proves that (A.3) holds.
Since $S(\pi)$ is countable,
the existence of the sequences $(t_k)$ and $(\xi^k)$ for every $t \in S(\pi)$ as described in the statement of the proposition follows.

We finally prove that (A.4) holds as well.
Fix $t\in T(\pi) \setminus\{ 1\}$, so that $ \widehat\pi_t=t$.
We have to show that $\dot\pi_t(a)=0$ for every $a\in A^*_{\geq 2}$.
Since $t \in T(\pi)$,
there is a nonincreasing sequence $(t_k)$ that converges to $t$ such that $\widehat\pi_{t_k-} = t_k$ for every $k$.
For the same reason,
for every $\ep > 0$ there is $k_0 \in \dN$ and $\delta > 0$ such that for every $k \geq k_0$
and every $t' \in [t_k,t_k+\delta) \cap S(\pi^k)$ we have $p(\xi^k_{t'}) < \ep$.
Indeed, otherwise there is $\ep > 0$ such that for every $k_0 \in \dN$ and every $\delta > 0$ there is $k \geq k_0$ and
$t' \in [t_k,t_k+\delta) \cap S(\pi^k)$ for which $p(\xi^k_{t'}) \geq \ep$.
But then, letting $k_0$ go to infinity and $\delta$ go to 0,
we deduce that $t \in S(\pi)$ and $p(\xi_t) \geq \ep$, a contradiction.

For every mixed action profile $\xi$ that satisfies $p(\xi) < \ep$,
we have $\xi^i < \ep$ for every $i$,
and therefore
\[ \xi(a) = \left(\prod_{\{ i \colon a^i = Q^i\}} \xi^i\right)\cdot\left(\prod_{\{ i \colon a^i = C^i\}} (1-\xi^i)\right) \leq \frac{\ep}{1-\ep} p(\xi),
\ \ \ \forall a \in A^*_{\geq 2}. \]

We deduce that
for every $\ep > 0$ there is $k_0 \in \dN$ and $\delta > 0$ such that for every $k \geq k_0$
and every $t' \in (t_k,t_k+\delta) \cap  S(\pi^k)$,
we have $\xi_{t'}(a) \leq \frac{\ep}{1-\ep} p(\xi_{t'})$ for every $a \in A^*_{\geq 2}$.
This implies that for every $t' \in (t_k,t_k+\delta) \cap  (T(\pi^k) \cup S(\pi^k))$
\[ \pi^k_{t'-}(a) - \pi^k_{t_k-}(a) \leq (t'-t_k) \frac{\ep}{1-\ep},
 \ \ \ \forall a \in A^*_{\geq 2}, \forall t' \in (t_k,t_k+\delta) \cap (T(\pi^k) \cup S(\pi^k)). \]
Since this inequality holds for every $\ep > 0$,
we deduce that $\dot\pi_t(a)=0$ for every $a\in A^*_{\geq 2}$.
\end{proof}

\subsection{The Payoff Path}

Let $\pi$ be an absorption path.
For every $0 \leq t < 1$ and every $a \in A^*$,
the difference $\pi_{1}(a) - \pi_{t}(a)$ is the probability that the play terminates
by the action profile $a$ in the interval $(t,1]$.
Since the probability of absorption
in $[t,1]$
is $1-\widehat\pi_{t}$, the expected payoff after absorption probability $t$ is given by the formula
\begin{equation}
\label{equ:payoff}
\gamma_{t}(\pi) :=
\left\{ \begin{array}{ll}
\frac{\sum_{a \in A^*} \left(\pi_{1}(a) - \pi_{t}(a)\right) r(a)}{1-\widehat\pi_{t}},& \mbox{ if } \widehat\pi_{t}<1,\\
\vec 0,&\mbox{ if } \widehat\pi_{t}=1.
\end{array}\right.
\end{equation}
We call the function $\gamma(\pi) :[0,1] \to \dR^I$
the \emph{payoff path}.

\begin{remarks}
    \label{remgamma}
    \begin{enumerate}
        \item  Payoff paths take their values in $[-M,M]^I$, where $M=\|r(a)\|_\infty$.
        \item
         Note that $\gamma_{0-}(\pi)=\sum_{a\in\A^*}\pi_1(a)r(a)$ is the expected payoff under $\pi$ in the game.
         The value of $\gamma_t(\pi)$ is irrelevant when $\widehat\pi_{t}=1$, because, in this case, the game is already over at $t$.
        \item
        For every absorbing strategy profile $x$, we have
        \[ \gamma_{t_n-}(\pi^x) = \gamma_n(x), \ \ \ \forall n \in \dN,\]
        where the absorption path $\pi^x$ is defined in Remark~\ref{remark:strategy}, and $t_n = \prob_x(\theta < n)$.
        This equality reflects the equivalence between each strategy profile $x$ and the absorption path $\pi^x$.
\item When $T(\pi)=[0,1]$, the expression for the payoff path simplifies to
            \begin{equation}
            \label{gamma t} \gamma_t(\pi)=\frac{\sum_{i\in I}\left(\pi_1(Q^i,C^{-i})-\pi_t(Q^i,C^{-i})\right)r(Q^i,C^{-i})}{1-t}
            \end{equation}
Then we have for every $0\leq s< t< 1$,
            \[ (1-t)\gamma_t=(1-s)\gamma_s+\sum_{i\in I}(\pi_s(Q^i,C^{-i})-\pi_t(Q^i,C^{-i})r(Q^i,C^{-i}).\]
Hence, the function $t\mapsto\gamma_t$ solves the differential equation
            \begin{equation}
            \label{odegamma}
            (1-t)\dot\gamma_t=\gamma_t-\sum_{i\in I}\dot\pi_t(Q^i,C^{-i})r(Q^i,C^{-i}), \; t\in[0,1).
            \end{equation}
        \item Let $(\pi^k)_{k \in \dN}$ be a sequence of  absorption paths that converges to a limit $\pi$.
        Then,
        \[ \gamma_t(\pi) = \lim_{k \to \infty} \gamma_t(\pi^k),\]
        for all
        $t\in [ 0,1)$
        where $\pi$ is continuous.
    \end{enumerate}
\end{remarks}

We now adapt the definition of sequential $\ep$-perfectness to absorption paths.

\begin{definition}
Let $\ep \geq 0$.
Player~$i$ is \emph{sequentially $\ep$-perfect}
at the absorption path $\pi$ if the following conditions hold:

\begin{enumerate}
\item[(SP.1)] For all $t\in S(\pi)$ such that
$\widehat\pi_{t}<1$,
player~$i$ is $\ep$-perfect at the mixed action profile $\xi_{t}$
in the strategic-form game $G_\Gamma(\gamma_{t}(\pi))$.
\item[(SP.2)] For every $t \in T(\pi) \setminus\{ 1\}$,
\begin{itemize}
    \item[(a)] $\gamma_{t}^i(\pi) \geq r^i(Q^i,C^{-i}) - \ep$, and
    \item[(b)]  if $\dot \pi_t(Q^i,C^{-i}) > 0$,
    then
$\gamma_{t}^i(\pi) \leq r^i(Q^i,C^{-i}) + \ep$.
\end{itemize}
\end{enumerate}
An absorption path $\pi$ is \emph{sequentially $\ep$-perfect}
if all players are sequentially $\ep$-perfect at $\pi$.
\end{definition}

In words, an absorption path is sequentially $\ep$-perfect if
(i) whenever the players play in discrete time ($t \in S(\pi)$),
the mixed action that they play is $\ep$-perfect in the one-shot game induced by the continuation payoff,
and (ii) whenever the players play in continuous time ($t \in T(\pi)$),
it cannot be that by quitting a player will gain more than $\ep$, and
a player does not quit with positive rate if her continuation payoff is higher by more than $\ep$ than
her payoff if she quits alone.

It follows by the definition of $\pi^x$ (see Remark~\ref{remark:strategy}),
that player~$i$ is sequentially $\ep$-perfect at an absorbing strategy profile $x$,
if and only if she is sequentially $\ep$-perfect at the absorption path $\pi^x$.

We shall see now that standard continuity arguments show that a
limit of sequentially $\ep$-perfect absorption paths as $\ep$ goes to 0 is a sequentially $0$-perfect absorption path.
\begin{proposition}
\label{lemma:0perfect} Let $(\pi^k)_{k \in \dN}$ be a sequence of
absorption paths that converges to a limit $\pi$,
let $(\ep^k)_{k \in \dN}$ be a sequence of non-negative reals that converges to 0,
and let $i \in I$.
If for every $k \in \dN$ player~$i$ is sequentially $\ep^k$-perfect at the absorption path $\pi^k$,
then player~$i$ is sequentially 0-perfect at the absorption path $\pi$.
\end{proposition}

\begin{proof}
Fix $t\in S(\pi)$. We prove that in this case (SP.1) holds with $\ep = 0$.
Since $\pi^k \Rightarrow \pi$, following Proposition \ref{AP:compact}
we can find  a sequence $(t_k)_{k\in\N}$, with $t_k\in S(\pi^k)$ for all $k \in \dN$,
that converges to $t$ and such that $\xi_{t} = \lim_{k \to \infty} \xi^k$,
where $\xi^k$ satisfies Eq.~\eqref{deltaxi} at $t_k$ for $\pi^k$, for all $k \in \dN$.
Remark~\ref{remgamma}(5) implies that $\gamma_{t_k}(\pi^k) \to \gamma_t(\pi)$.
By definition, if player~$i$ is sequentially $\ep^k$-perfect at $\pi^k$,
then she is $\ep^k$-perfect at the mixed action profile $\xi^k$ in the strategic-form game $G_\Gamma(\gamma_{t_k}(\pi^k))$.
As discussed in Section~\ref{section:perfect},
it follows that player~$i$ is 0-perfect at $\xi_{t}$ in the strategic-form game
$G_\Gamma(\gamma_{t}(\pi))$,
i.e., (SP.1) holds with $\ep = 0$.

Now let $t\in T(\pi) \setminus\{ 1\}$.
We will prove that (SP.2.a) holds with $\ep = 0$.
Let $(t_k)$ be a nonincreasing sequence of times converging to $t$, such that $\pi^k_{t_k-}\to\pi_t$.
This implies that $\gamma_t(\pi) = \lim_{k \to \infty} \gamma_{t_k}(\pi^k)$.
As in the proof of Proposition \ref{AP:compact},
we can choose this sequence in a way that $\widehat\pi^k_{t_k-}=t_k$ for all $k\in\dN$.
Following Remark~\ref{const}(7), this implies that,
for each $k \in \dN$ there are only two possibilities:
either $t_k\in T(\pi^k)$ or $t_k\in S(\pi^k)$.

Suppose first that $t_k\in T(\pi^k)$ for every $k \in \dN$ large enough.
Then (SP.2.a), applied to $\pi^k$, yields
\[ \gamma^i_{t_k}(\pi^k)\geq r^i(Q^i,C^{-i})-\ep^k, \]
 and, letting $k$ go to $+\infty$, we obtain that
(SP.2.a) with $\ep=0$ holds for $\pi$ at $t$.

Next let us suppose the existence of a subsequence of $(\pi^k)$ such that $t_k\in S(\pi^k)$ for every $k \in \dN$.
By assumption we have
\begin{eqnarray}
\label{ek}
r^i(Q^i,\xi^{k,-i}) \leq
\gamma_{t_k}^i(\pi^k)+ \ep^k,
\end{eqnarray}
As in the proof of Proposition \ref{AP:compact}, the sequence $(p(\xi^k_{t_k}))$ vanishes when $k$ tends to $+\infty$, or, equivalenty,  $\xi^k_{t_k}\to\vec 0$.
The result follows by letting $k$ go to  $+\infty$ in Eq.~\eqref{ek}.

The proof that (SP.2.b) holds with $\ep=0$ is similar,
hence (SP.2) holds for every $t \in T(\pi)$ such that $\pi^k_t\to\pi_t$.
For $t$ such that $\pi^k_t$ does not converge to $\pi_t$, (SP.2) holds by the right-continuity of $\pi$.
\end{proof}

The following result relates the concepts of $\ep$-equilibria and
sequential 0-perfect absorption paths.

\begin{theorem}
\label{theorem:6}
Let $\Gamma$ be a quitting game that does not possess an $\ep$-equilibrium
under which the game terminates with probability 1 in the first stage.
The game admits an $\ep$-equilibrium for every $\ep > 0$,
if and only if there is a sequentially 0-perfect absorption path.
\end{theorem}

\begin{proof}
Theorem~\ref{theorem:simon} and Proposition~\ref{lemma:0perfect}
imply that if the game admits an $\ep$-equilibrium for every $\ep > 0$,
then there is a sequentially 0-perfect absorption path.
Regarding the converse implication, let $\pi$ be  a sequentially 0-perfect absorption path.
In the proof of Proposition~\ref{prop:convergence} we constructed
a sequence $(x^k)$ of strategy profiles such that $\pi^{x^k}\to \pi$.
In the notations of the proof of Proposition~\ref{prop:convergence},
$\sup_{n \in \dN}\| \gamma_{s^k_n}(\pi^{x^k}) - \gamma_{s^k_n}(\pi)\|_\infty \to 0$,
which implies that $x^k$ is an $\ep^k$-equilibrium for every $k$, with $\ep^k \to 0$.
\end{proof}

\bigskip

Theorem~\ref{theorem:6} is related to
Gobbino and Simon (2020),
who separated the dynamics of
the sequence $(\gamma_n(x))_{n \in \dN}$, where $x$ is an absorbing sequentially $\ep$-perfect strategy profile,
into ``large'' motion (the discrete part of the absorption path)
and ``small'' motion (the continuous part of the absorption path).

\section{Continuous Equilibria}
\label{section:continuous}

An absorption path $\pi$ is \emph{continuous} if it does not contain discrete-time aspects;
that is, if $T(\pi) = [0,1]$.
When $\pi$ is continuous, $\sum_{a \in A^*_{\geq 2}} \pi_1(a) = 0$,
yet the converse need not hold.
To simplify terminology,
we use the term \emph{continuous equilibria} for sequentially 0-perfect continuous absorption paths.

In this section we provide a sufficient condition for the existence of
a continuous equilibrium.
To present the sufficient condition,
it is convenient to normalize the payoffs and assume w.l.o.g.~that $r^i(Q^i,C^{-i}) = 0$ for each $i \in I$.

\begin{definition}
Let $R$ be an $(n \times n)$-matrix, and let $q \in \dR^n$.
For each $i$, $1 \leq i \leq n$, denote by $R^i$ the $i$'th column of $R$.
The \emph{linear complementarity problem} $\lcp(R,q)$
is the following problem:
\begin{eqnarray}
\nonumber
\hbox{Find}&&w \in \dR^n_{+}, \hbox{ and } z = (z_0,z_1,\cdots,z_n) \in \Delta(\{0,1,\cdots,n\}),\label{lpc}\\
\hbox{such that}
&& w = z_0q + \sum_{i =1}^n z_i R^i,\\
&&z_i = 0 \hbox{ or } w_i = 0, \ \ \ \forall i \in \{1,2,\ldots,n\}.
\nonumber
\end{eqnarray}
\end{definition}

A matrix $R$ is a \emph{$Q$-matrix} if for every $q \in \dR$ the problem $\lcp(R,q)$ has at least one solution.

Let $\Gamma$ be a quitting game,
and denote by $R(\Gamma)$ the $(|I| \times |I|)$ matrix $(r^i(Q^j,C^{-j}))_{i,j \in I}$.
Solan and Solan (2020) proved that if $R(\Gamma)$ is not a $Q$-matrix,
then $\Gamma$ has a stationary 0-equilibrium.
Here we handle the case where $R(\Gamma)$, as well as all its principal minors,
are $Q$-matrices.

\begin{theorem}
\label{theorem:continuous}
If $R(\Gamma)$ and all its principal minors are $Q$-matrices,
then there exists continuous equilibrium.
\end{theorem}

\begin{remark}
Theorem~\ref{theorem:continuous} is not tight:
there may be continuous equilibria when its condition is not satisfied.
Indeed, it may be that the restriction of $R(\Gamma)$ to a subset of players satisfies the condition of Theorem~\ref{theorem:continuous},
and therefore there is a continuous equilibrium $\pi$ for the subgame that involves those players
(when all other players are restricted to always continue),
and it may further happen that the other players obtain high payoffs along this absorption path.
In such a case, all players are sequentially 0-perfect at $\pi$

We do not know whether the existence of a continuous equilibrium along which all players quit with positive probability
implies that $R(\Gamma)$ and all its principal minors are $Q$-matrices.
\end{remark}

\begin{proof}[Proof of Theorem~\ref{theorem:continuous}]

\noindent\textbf{Step 1:} Convex combinations in the non-negative orthant.

We will show here that for every nonempty subset $J \subseteq I$ of players
there is a probability distribution $z \in \Delta(J)$ that satisfies
\begin{eqnarray}
\label{equ:131}
\sum_{i \in J} z_i r^j(Q^i,C^{-i}) &\geq& 0, \ \ \ \forall j \in J,\\
\label{equ:132}
\sum_{i\in J}z_ir^j(Q^i,C^{-i}) &=& 0 \hbox{ for at least one }j\in J.
\end{eqnarray}
The assumption that $R = R(\Gamma)$ and all its principal minors are
$Q$-matrices is used only in this step of the proof.

Fix $i_0 \in J$ and
let $\widehat q \in \dR^J$ be the vector that is defined by
\[ \widehat q_{i_0} := -1, \ \ \ \widehat q_i := 0 \ \ \ \forall i \in J \setminus \{i_0\}. \]
The matrix $\widehat R:=(r^i(Q^j,C^{-j}))_{i,j\in J}$ is a principal minor of  $R$.
Therefore, the linear complementarity problem $\lcp(\widehat R,\widehat q)$ has a solution $(\widehat w,\widehat z)$.
Since $\widehat q_{i_0} < 0$, it cannot be that $\widehat z_0 = 1$.
If $i_0$ is the only player $i \in J$ such that $\widehat z_i > 0$,
then, since $r^{i_0}(Q^{i_0},C^{-i_0}) = 0$ and $\widehat q_{i_0} < 0$,
we have $\widehat z_{i_0} = 1$.
Otherwise, there is $i_1 \in J \setminus\{i_0\}$ such that $\widehat z_{i_1} > 0$,
and consequently $\widehat w_{i_1} = 0$.

Define $z_i := \frac{\widehat z_i}{1-\widehat z_0}$ for each $i \in J$.
Since $\widehat w_i \geq 0$ and $\widehat q_i \leq 0$ for every $i \in J$,
and since $\widehat w$ is a convex combination of $\widehat q$ and $\sum_{i \in J} z_i r(Q^i,C^{-i})$,
it follows that Eq.~\eqref{equ:131} holds.
If $z_{i_0} = 1$, then Eq.~\eqref{equ:132} holds with $j=i_0$.
Otherwise,
since $\widehat w_{i_1}=\widehat q_{i_1}=0$, we have
$\sum_{i\in J}z_ir^{i_1}(Q^i,C^{-i})=0$, and Eq.~\eqref{equ:132} holds with $j=i_1$.

\medskip
\noindent\textbf{Step 2:} Viability theory.

For every $z\in\Delta(I)$ denote $z\cdot R:=\sum_{i\in I}z_iR^i$,
and let $Y$ be the boundary of $\R^I_+$.
For every $q\in Y$, set
\[ F(q) := \{ z \in \Delta(I) \colon z_i > 0 \ \ \Rightarrow \ \ q_i = 0, \ \ \
(z \cdot R)_i \geq 0 \hbox{ whenever } q_i = 0\}.
\]
We will show that there exists a measurable function $z : [t_0,1] \to \Delta(I)$
such that for every $t \in [t_0,1]$ we have
(a) $q(t)\in Y$ and
(b) $z(t) \in F(q(t))$.

The set-valued function $F$ is upper semi-continuous with convex values,
and by Step~1 it has nonempty values.
For every $q \in Y$ denote by $T_Y(q)$ the tangent cone at $q$:
\[ T_Y(q):=\{ d\in\R^I, q+\delta d\in Y \mbox{ for all }\delta>0 \mbox{ small}\}.\]
A careful analysis of the tangent cone shows that
$\frac{\delta}t z\cdot R+(1-\frac {\delta}t) q\in T_Y(q)$ for every $z$
satisfying Eqs.~\eqref{equ:131}--\eqref{equ:132} and $\delta>0$ small enough,
where $J = \{i \in I \colon q_i = 0\}$.

Fix $(q_0,t_0)\in Y\times (0,1)$.
For every measurable function $z : [t_0,1] \to \Delta(I)$, consider the following controlled dynamic:
\begin{equation}
\label{qz}
\left\{\begin{array}{l}
\dot q(t)= \frac 1t(z(t)\cdot R - q(t)), \ \ \ \forall t\in[t_0,1],\\
q(t_0)=q_0.
\end{array}\right.
\end{equation}
The set $Y$ is closed,
and the set-valued function $F$ is upper-semicontinuous with nonempty, closed, and convex values.
By the classical Viability Theorem (Aubin, 1991, Theorem 3.3.4) it follows that
there exists a measurable function $z : [t_0,1] \to \Delta(I)$
such that
(a) and (b) above hold
for every $t \in [t_0,1]$.

\medskip
\noindent\textbf{Step 3:} Constructing a continuous equilibrium.

Fix an arbitrary $q_0\in Y$.
For every $n \in \dN$ let $(q^n,z^n)$ be a solution of Eq.~\eqref{qz} with $q^n_0 = q_0$ and $t_0 = \frac{1}{n}$,
such that $q^n(t) \in Y$ and $z^n(t) \in F(q^n(t))$ for every $t \in [\frac{1}{n},1]$.
Define $\pi^{n}\in\A$ by
\begin{equation}
\label{equ:34}
\dot\pi^{n}_t(Q^i,C^{-i})=z^{n}_i(1-t), \ \ \ \forall t\in[0,1-\tfrac{1}{n}), \ \ \ \forall i\in I,
\end{equation}
and an arbitrary continuous evolution on $[1-\frac{1}{n},1]$.
By definition, $\pi^n$ is a continuous absorption path.
Eq.~\eqref{qz} implies that, for all $0 \leq t \leq 1 - \frac{1}{n}$,
\[
(1-t)q^n(1-t)-\frac 1nq_0
=\int_{\frac 1n}^{1-t}z^n(s)\rmd s\cdot R=\int_{t}^{1-\frac 1n}z^n(1-s)\rmd s\cdot R.\]
In addition, for every $t \in [0,1-\tfrac{1}{n}]$,
\begin{eqnarray*}
\gamma_{t}(\pi^n)
&=& \frac{1}{1-t} \int_t^1 z^n(1-s) \rmd s \cdot R\\
&=& \frac{1}{1-t} \int_t^{1-1/n} z^n(1-s) \rmd s \cdot R + \frac{1}{1-t} \int_{1-1/n}^1 z^n(1-s) \rmd s \cdot R\\
&=& q^n(1-t) - \frac{q_0}{(1-t)n}  + \frac{1}{1-t} \int_{1-1/n}^1 z^n(1-s) \rmd s \cdot R.
\end{eqnarray*}
It follows that
\[ \| \gamma_{t}(\pi^n) - q^n(1-t)\|_\infty \leq \frac{2\|R\|_\infty}{(1-t)n}, \ \ \ \forall n \in \dN, \forall t \in [0,1-\tfrac{1}{n}]. \]
Let $\pi$ be an accumulation point of $(\pi^n)$, and assume w.l.o.g.~that $\pi^n \Rightarrow \pi$.
Since $\pi^n$ is continuous, so is $\pi$.
Consequently, for every $t \in [0,1)$ the limit $\lim_{n \to \infty} q^n(1-t)$ exists and is equal to $\gamma_t(\pi)$.
Since $q^n(1-t) \in Y$ for every $t \in [0,\frac{1}{n}]$,
we deduce that $\gamma_t(\pi) \in Y$ for every $t \in [0,1)$, and therefore (SP.2.a) with $\ep =0$ holds for each $i \in I$.

We turn to prove that (SP.2.b) holds as well.
Fix $i \in I$ and let $t \in [0,1)$ be such that $\dot\pi_t(Q^i,C^{-i}) > 0$.
Then there exists a sequence $(t_n)_{n \in \dN}$ such that $\lim_{n \to \infty} t_n = t$ and $\dot \pi^n_{t_n}(Q^i,C^{-i}) > 0$
for every $n$ sufficiently large.
This implies that  for every $n$ sufficiently large we have $z^n_i(1-t_n) > 0$,
and therefore $q^n_i(1-t_n) = 0$.
By taking the limit as $n$ goes to infinity we deduce that $\gamma^i_t(\pi) = 0$,
and (SP.2.b) indeed holds.

Since Condition~(SP.2) holds for $\pi$, and since $i$ is arbitrary,
$\pi$ is sequentially 0-perfect.
\end{proof}

\bigskip

When $\pi$ is a continuous equilibrium,
we can assign to each $t \in [0,1)$ the set of players who quit with positive rate at $t$.
In the next two examples,
$[0,1)$ is divided into countably many intervals, and
a single player quits with positive rate in each interval.
We therefore describe $\pi$ by a list of pairs $(i_k,p_k)_k$, where $i_k$ is a player and $p_k \in (0,1]$:
under $\pi$, player~$i_0$ quits in the interval $[0,p_0)$,
player~$i_1$ quits in the interval $[p_0,p_0 + (1-p_0)p_1)$,
and so on.
Thus, $(i_k)_k$ indicates the order by which the players quit,
and $p_k$ indicates the probability by which player~$i_k$ quits
in the $k$'th interval,
given that the game
did not terminate before.
Since the play eventually absorbs, $\sum_k p_k = \infty$,
yet it might be that the index set of $k$ is not $\dN$, as happens in Example~\ref{example:five} below.

\begin{example}
\label{example:ftv:0}
Suppose that there are three players, $r(Q^1,C^2,C^3) = (0,2,-1)$, $r(C^1,Q^2,C^3) = (-1,0,2)$, and $r(C^1,C^2,Q^3) = (2,-1,0)$.
Games that have these payoffs was studied by Flesch, Thuijsman, and Vrieze (1997)
and Solan (2003).
The corresponding matrix $R$ and all its principle minors are $Q$-matrices, hence a
continuous equilibrium exists.
One such equilibrium
is the one were the sequence $(i_k,p_k)_k$ is:
\begin{equation}
\label{equ:891}
(1,\frac{1}{2}), (2,\frac{1}{2}), (3,\frac{1}{2}), (1,\frac{1}{2}), (2,\frac{1}{2}), (3,\frac{1}{2}), (1,\frac{1}{2}), (2,\frac{1}{2}), (3,\frac{1}{2}), \ldots.
\end{equation}

\bigskip
In fact, it can be shown that all
continuous equilibria in this example
can be obtained from the one in Eq.~\eqref{equ:891} by starting the period at any
$t \in [0,\frac{7}{8}]$
(instead of at $t=0$).
\end{example}

The following example shows that continuous equilibria
even when periodic,
may exhibit a wild behavior.

\begin{example}
\label{example:five}
Suppose that there are five players,
$r(Q^1,C^2,C^3,C^4,C^5) = (0,2,-\frac{1}{2},1,-1)$,
$r(C^1,Q^2,C^3,C^4,C^5) = (-\frac{1}{2},0,2,1,-1)$,
$r(C^1,C^2,Q^3,C^4,C^5) = (2,-\frac{1}{2},0,1,-1)$,
$r(C^1,C^2,C^3,Q^4,C^5) = (-1,-2,-3,0,\frac{10}{7})$,
and $r(C^1,C^2,C^3,C^4,Q^5) = (2,\frac{7}{2},\frac{47}{8},\frac{5}{2},0)$.
It is a bit tedious but not difficult to show that the corresponding matrix $R$ and all its principle minors are $Q$-matrices,
and therefore a
continuous equilibrium exists.

In this example there are many periodic
continuous equilibria $(i_k,p_k)_k$.
In fact, for every $l \in \dN$ there is
such an equilibrium
 with period $3l+2$,
where the sequence $(i_k)_{k=1}^{3l+2}$ is
$(1,2,3,1,2,3,\ldots,1,2,3,4,5)$.

Yet there is also a
continuous equilibrium that has this structure for $l=\infty$:
\[ (1,\frac{1}{4}), (2,\frac{1}{6}), (3,\frac{1}{20}), (1,\frac{1}{76}), (2,\frac{1}{300}), (3,\frac{1}{598}), \ldots, (4,\frac{1}{2}), (5,\frac{1}{2}). \]
We do not know whether there exist games where there is a
continuous equilibrium but none that is periodic with a finite period.
An algorithm for calculating the
union of the range of all payoff paths that correspond to continuous equilibria
is described in
Ashkenazi-Golan, Krasikov, Rainer, and Solan (2020).
\end{example}

\section{Discussion}
\label{section:discussion}

The behavior of players in dynamic games in general, and quitting games in particular, may be complex.
It might be that in some stage, the players mix their actions,
knowing that the set of players who will terminate the game will be random.
It might also happen that some player wants to quit,
but she wants to guarantee that no other player knows when she quits,
to avoid the outcome where she quits with someone else.
While in discrete time a player cannot guarantee that no other player will be able to quit with her,
in continuous time this can be done.
Equilibrium behavior in quitting games may exhibit both types of behavior:
periods of discrete-time behavior, when players quit with positive probability,
and periods of continuous-time behavior, when players quit at a given rate.

The concepts of discrete-time strategies and continuous-time strategies can capture only one of the
two possible behaviors described above.
In this paper we introduced an alternative representation of strategy profiles in quitting games,
called absorption paths,
which allows to describe both behaviors.
Though it is not known whether all quitting games have $\ep$-equilibria,
we showed that \emph{if} an $\ep$-equilibrium exists for every $\ep > 0$,
\emph{then} there exists a sequentially 0-perfect absorption path.
This result shows that,
the reason for having games that possess $\ep$-equilibria for every $\ep > 0$ but no 0-equilibria,
is that the nature of discrete time does not allow players to completely
hide the stage in which they quit,
thereby allowing other players to quit simultaneously with them
(albeit with small probability) and make a low profit.

The space of absorption paths $\A$ is compact,
and the function that assigns to every absorption path its payoff path is continuous.
It is not difficult to show that $\A$ is contractible.
We do not know whether these properties can be used to prove the existence
of an $\ep$-equilibrium in some family of quitting games.

\end{document}